\documentclass[11pt]{amsart}
\usepackage{latexsym,amsmath,amssymb,epsfig,amsthm}
\usepackage{graphicx}
\usepackage{tikz}
\usepackage[enableskew,vcentermath]{youngtab}
\usepackage{hyperref}
\hypersetup{colorlinks=false}
\input epsf
\newtheorem{theorem}{Theorem}[section]
\newtheorem{proposition}[theorem]{Proposition}
\newtheorem{corollary}[theorem]{Corollary}

\newtheorem{definition}[theorem]{Definition}

\newtheorem{remark}[theorem]{Remark}

\newcommand{\sd}{{\rm sd}}

\newcommand{\cC}{{\mathcal C}}

\newcommand{\lL}{{\mathcal L}}

\newcommand{\RR}{{\mathbb R}}

\newcommand{\NN}{{\mathbb N}}

\newcommand{\sm}{{\smallsetminus}}
\begin{document}
\title[Barycentric subdivisions of cubical complexes]
{Face numbers of barycentric subdivisions of cubical 
complexes}

\author{Christos~A.~Athanasiadis}

\address{Department of Mathematics\\
National and Kapodistrian University of Athens\\
Panepistimioupolis\\
15784 Athens, Greece}
\email{caath@math.uoa.gr}

\date{December 18, 2020}

\begin{abstract}
The $h$-polynomial of the barycentric subdivision 
of any $n$-dimensional cubical complex with nonnegative 
cubical $h$-vector is shown to have only real roots and 
to be interlaced by the Eulerian polynomial of type 
$B_n$. This result applies to barycentric subdivisions 
of shellable cubical complexes and, in particular, to 
barycentric subdivisions of cubical convex polytopes 
and answers affirmatively a question of Brenti, 
Mohammadi and Welker.
\end{abstract}

\maketitle

\section{Introduction}
\label{sec:intro}

A fundamental problem in algebraic and geometric 
combinatorics is to characterize, or at least obtain 
significant information about, the face enumerating
vectors of triangulations of various topological 
spaces, such as balls and spheres \cite{StaCCA}. 
Face enumerating vectors are often presented in
the form of the $h$-polynomial (see 
Section~\ref{sec:enu} for definitions). Properties 
such as unimodality, log-concavity, 
$\gamma$-positivity and real-rootedness have been of 
primary interest \cite{Ath18, Bra15, Bre94b, Sta89}. 
One expects that the `nicer' the triangulation is 
combinatorially and the space being triangulated is 
topologically, the better the behavior of the 
$h$-polynomial is.

Following this line of thought, Brenti and 
Welker~\cite{BW08} considered an important and well  
studied triangulation in mathematics, namely the 
barycentric subdivision. They studied the 
transformation of the $h$-polynomial of a simplicial 
complex $\Delta$ under barycentric subdivision and 
showed that the resulting $h$-polynomial has only 
real roots (a property with strong implications) for 
every simplicial complex $\Delta$ with nonnegative 
$h$-polynomial. They asked \cite[Question~3.10]{BW08} 
whether the $h$-polynomial of the barycentric 
subdivision of any convex polytope has only real 
roots, suspecting an affirmative answer (see 
\cite[p.~105]{MW17}). This question was 
raised again by Mohammadi and 
Welker~\cite[Question~35]{MW17} and, as is 
typically the case in face enumeration, 
it is far more interesting and more challenging for 
general polytopes and polyhedral complexes, than it 
is for simplicial polytopes 
and simplicial complexes. Somewhat surprisingly, no 
strong evidence has been provided in the 
literature that such a result may (or may not) hold 
beyond the simplicial setting. One should also note 
that barycentric subdivisions of boundary complexes 
of polytopes form a special class of flag 
triangulations of spheres and that the real-rootedness 
property fails for the $h$-polynomials of this more 
general class of triangulations in dimensions higher 
that four \cite{Ga05}. At present, it is unclear
where the borderline between positive and negative 
results lies.

Mohammadi and Welker (based on earlier discussions 
with Brenti) suggested the class of cubical polytopes 
as another good test case; see \cite[p.~105]{MW17}. Cubical 
complexes and polytopes are important and mysterious 
objects with highly nontrivial combinatorial properties
(see, for instance, \cite{Ad96, BBC97, Jo93, JZ00}). 
They have been studied both for their own independent 
interest, and for the role they play in other 
areas of mathematics. Given the intricacy of their 
combinatorics, it comes as no surprise that the 
question of Brenti and Welker turns out to be more 
difficult for them than for simplicial complexes. The 
following theorem provides 
the first general positive result on this question, 
since \cite{BW08} appeared, and suggests that an 
affirmative answer should be expected at least for 
broad classes of nonsimplicial convex polytopes (or 
even more general cell complexes and posets).
\begin{theorem} \label{thm:main}
The $h$-polynomial of the barycentric subdivision 
of any shellable cubical complex has only real roots. 
In particular, barycentric subdivisions of cubical 
polytopes have this property.
\end{theorem}

The case of cubical polytopes was also studied recently 
by Hlavacex and Solus~\cite{HS20+}. Using the concept 
of shellability and the theory of interlacing 
polynomials, they gave an affirmative answer 
for cubical complexes which admit a special type of
shelling and applied their result to certain 
families of cubical polytopes, such as cuboids, 
capped cubical polytopes and neighborly cubical 
polytopes. 

The proof of the aforementioned result of~\cite{BW08}
applies a theorem of Br\"and\'en~\cite{Bra06} on the
subdivision operator \cite[Section~3.3]{Bra15} to a
formula for the $h$-polynomial of the barycentric 
subdivision of a simplicial complex (see 
Remark~\ref{rem:BW-formula}). The proof of 
Theorem~\ref{thm:main} is motivated by the proof of 
the result of~\cite{BW08}, given and extended to the 
setting of uniform triangulations of simplicial 
complexes in~\cite{Ath20+} (the latter was partially 
motivated by \cite[Example~8.1]{Bra15}). To explain 
further, we let $h(\Delta, x) = \sum_{i=0}^{n+1} 
h_i(\Delta) x^i$ denote the $h$-polynomial and 
$\sd(\Delta)$ denote 
the barycentric subdivision of an $n$-dimensional 
simplicial complex $\Delta$. As already shown 
in~\cite{BW08}, there exist polynomials with 
nonnegative coefficients $p_{n,k}(x)$ for $k \in 
\{0, 1,\dots,n+1\}$, which depend only on $n$ and $k$, 
such that 
\begin{equation} \label{eq:BW}
h(\sd(\Delta), x) \ = \ \sum_{k=0}^{n+1} h_k(\Delta) 
                        p_{n,k}(x)
\end{equation}
for every $n$-dimensional simplicial complex $\Delta$.
For every $n \in \NN$, the polynomials $p_{n,k}(x)$ can 
be shown \cite[Example~8.1]{Bra15} to have only real 
roots and to form an interlacing sequence. This implies 
that their nonnegative linear combination 
$h(\sd(\Delta), x)$ also has only real (negative) 
roots and that it is interlaced by $p_{n,0}(x)$, which
equals the classical $(n+1)$st Eulerian polynomial 
$A_{n+1}(x)$ \cite[Section~1.4]{StaEC1}. The interlacing 
condition implies that the roots of $h(\sd(\Delta), x)$
are not arbitrary, but rather that they lie in certain 
intervals that depend only on the dimension $n$, formed
by zero and the roots of $A_{n+1}(x)$. The polynomial 
$p_{n,k}(x)$ can be interpreted as the $h$-polynomial 
of the relative simplicial complex obtained from the 
barycentric subdivision of the $n$-dimensional simplex
by removing all faces lying on $k$ facets of the 
simplex \cite[Section~5]{Ath20+} 
\cite[Section~4.2]{HS20+}.  

This paper presents a similar picture for cubical 
complexes. We define (see Definition~\ref{def:pBnk}) 
polynomials $p^B_{n,k}(x)$ for $k \in \{0, 1,\dots,n+1\}$ 
as the $h$-polynomials of relative simplicial complexes 
obtained from the barycentric subdivision of the 
$n$-dimensional cube by removing all faces lying on 
certain facets of the cube and prove (see 
Theorem~\ref{thm:h-trans}) that Equation~(\ref{eq:BW}) 
continues to hold when $\Delta$ is replaced by an 
$n$-dimensional cubical complex $\lL$, $p_{n,k}(x)$ is 
replaced by $p^B_{n,k}(x)$ and the $h_k(\Delta)$ are 
replaced by the entries of the (normalized) cubical 
$h$-vector of $\lL$, introduced and studied by 
Adin~\cite{Ad96}. We provide recurrences (see 
Proposition~\ref{prop:pBnk}) for the polynomials 
$p^B_{n,k}(x)$ which guarantee that they form an 
interlacing 
sequence for every $n \in \NN$ and conclude that 
$h(\sd(\lL), x)$ has only real (negative) roots and 
that it is interlaced by the $n$th Eulerian polynomial 
$B_n(x)$ of type $B$ for every $n$-dimensional cubical
complex $\lL$ with nonnegative cubical $h$-vector (see
Corollary~\ref{cor:main}). This implies 
Theorem~\ref{thm:main}, since shellable cubical 
complexes were shown~\cite{Ad96} to have nonnegative 
cubical $h$-vector and boundary complexes of convex 
polytopes are shellable~\cite{BM71}. 

The main results of this paper apply to cubical 
regular cell complexes (equivalently, to cubical 
posets) and will be stated at this level of 
generality. What comes perhaps unexpectedly 
is the fact that the transformation of a cubical 
$h$-polynomial into a simplicial one can be so well 
behaved. Corollary~\ref{cor:main} has nontrivial 
applications to triangulations of simplicial 
complexes as well; see Remark~\ref{rem:simplicial}.

\section{Face enumeration of simplicial and cubical 
complexes} \label{sec:enu}

This section recalls some definitions and background 
on the face enumeration of simplicial and cubical 
complexes and their triangulations, and shellability. 
For more information and any undefined terminology, 
we recommend the books \cite{HiAC, StaCCA}. All 
cell complexes considered here are assumed to be 
finite. Throughout this paper, we set $\NN = 
\{0, 1, 2,\dots\}$ and denote by $|S|$ the cardinality 
of a finite set $S$.

\subsection{Simplicial complexes}
\label{sec:simplicial}

An $n$-dimensional \emph{relative simplicial complex} 
\cite[Section~III.7]{StaCCA} is a pair $(\Delta, 
\Gamma)$, denoted $\Delta / \Gamma$, where $\Delta$ 
is an (abstract) $n$-dimensional simplicial complex 
and $\Gamma$ is a subcomplex of $\Delta$. The 
\emph{$f$-polynomial} of $\Delta / \Gamma$ is defined 
as 
\[ f(\Delta / \Gamma, x) \ = \ \sum_{i=0}^{n+1} f_{i-1} 
(\Delta / \Gamma) x^i, \]
where $f_j(\Delta / \Gamma)$ is the number of 
$j$-dimensional faces of $\Delta$ which do not belong 
to $\Gamma$. The \emph{$h$-polynomial} is defined as
\begin{eqnarray*} \label{def:simpl-h}
h(\Delta/\Gamma, x) & = & (1-x)^{n+1} f(\Delta/\Gamma, 
\frac{x}{1-x}) \ = \ \sum_{i=0}^{n+1} f_{i-1} (\Delta/\Gamma) 
\, x^i (1-x)^{n+1-i} \\ & = & \sum_{F \in \Delta/\Gamma} 
x^{|F|} (1-x)^{n+1-|F|} \ := \ \sum_{k=0}^{n+1} \, h_k 
(\Delta/\Gamma) x^k \nonumber
\end{eqnarray*}

\medskip
\noindent
and the sequence $h(\Delta/\Gamma) := (h_0(\Delta/\Gamma), 
h_1(\Delta/\Gamma),\dots,h_{n+1}(\Delta/\Gamma))$ is called 
the \emph{$h$-vector} of $\Delta/\Gamma$. Note that 
$f(\Delta / \Gamma, x)$ has only real roots if and only 
if so does $h(\Delta/\Gamma, x)$. When $\Gamma$ is empty,
we get the corresponding invariants of $\Delta$ and drop 
$\Gamma$ from the notation. Thus, for example,
$h(\Delta, x) = \sum_{k=0}^{n+1} h_k (\Delta) x^k$ is 
the (usual) $h$-polynomial of $\Delta$.

\smallskip
Suppose now that $\Delta$ triangulates an 
$n$-dimensional ball. Then, the boundary complex 
$\partial \Delta$ is a triangulation of an 
$(n-1)$-dimensional sphere and the \emph{interior 
$h$-polynomial} of $\Delta$ is defined as 
$h^\circ(\Delta, x) = h(\Delta / \partial \Delta, 
x)$. The following statement is a special case of 
\cite[Lemma~6.2]{Sta87}. 
\begin{proposition} \label{prop:hsymmetry} 
{\rm (\cite{Sta87})}
Let $\Delta$ be a triangulation of an 
$n$-dimensional ball. Let $\Gamma$ be a 
subcomplex of $\partial\Delta$ which is homeomorphic 
to an $(n-1)$-dimensional ball and $\bar{\Gamma}$ be 
the subcomplex of $\partial\Delta$ whose facets are 
those of $\partial\Delta$ which do not belong to 
$\Gamma$. Then,
\[ h(\Delta / \bar{\Gamma}, x) \ = \ x^{n+1} 
   h(\Delta / \Gamma, 1/x) . \]
Moreover, $h^\circ(\Delta, x) = x^{n+1} h(\Delta, 
1/x)$.
\end{proposition}

\subsection{Cubical complexes}
\label{sec:cubical}


A \emph{regular cell complex} \cite[Section 4.7]{OM}
is a (finite) collection $\lL$ of subspaces of a 
Hausdorff space $X$, called \emph{cells} or 
\emph{faces}, each homeomorphic to a closed unit 
ball in some finite-dimensional Euclidean space, 
such that: (a) $\varnothing \in \lL$; (b) the relative 
interiors of the nonempty cells partition $X$; 
and (c) the boundary of any cell in $\lL$ is a 
union of cells in $\lL$. The \emph{boundary complex} 
of $\sigma \in \lL$, denoted by $\partial \sigma$,
is defined as the regular cell complex consisting 
of all faces of $\lL$ properly contained in $\sigma$.
A regular cell complex $\lL$ is called \emph{cubical} 
if every nonempty face of $\lL$ is combinatorially 
isomorphic to a cube. A convex polytope is called 
\emph{cubical} if so is its boundary complex. 

Given a cubical complex $\lL$ of 
dimension $n$, we denote by $f_k(\lL)$ the number of 
$k$-dimensional faces of $\lL$. The cubical
$h$-polynomial was introduced and studied by 
Adin~\cite{Ad96} as a (well behaved) analogue of the 
(simplicial) $h$-polynomial of a simplicial complex. 
Following \cite[Section~4]{EH00}, we define the 
(normalized) \emph{cubical $h$-polynomial} of $\lL$ as 
\begin{equation} \label{def:cub-h}
(1+x) h(\lL, x) \ = \ 1 \, + \, \sum_{k=0}^n \, 
f_k(\lL) \, x^{k+1} \left( \frac{1-x}{2} \right)^{n-k} 
+ \ (-1)^n \, \widetilde{\chi} (\lL) x^{n+2} , 
\end{equation}
where $\widetilde{\chi} (\lL) = -1 + \sum_{k=0}^n 
(-1)^k f_k(\lL)$ is the reduced Euler characteristic 
of $\lL$ (the only difference from Adin's definition 
is that all coefficients of $h(\lL, x)$ have been 
divided by $2^n$ and, therefore, are not necessarily 
integers). We note that $h(\lL, x)$ is indeed a 
polynomial in $x$ of degree at most $n+1$. The 
(normalized) \emph{cubical $h$-vector} of $\lL$ is the 
sequence $h(\lL) = (h_0(\lL), h_1(\lL),\dots,h_{n+1}
(\lL))$, where $h(\lL, x) = \sum_{k=0}^{n+1} h_k(\lL) 
x^k$. 

\smallskip
Adin showed that $h(\lL, x)$ has nonnegative 
coefficients for every shellable cubical complex $\lL$ 
\cite[Theorem~5~(iii)]{Ad96} (his result is stated for
abstract cubical complexes with the intersection 
property, but the proof is valid without assuming the 
later). He asked whether the same holds whenever $\lL$ 
is Cohen--Macaulay 
\cite[Question~1]{Ad96}. The coefficient $h_k(\lL)$ is
known to be nonnegative for every Cohen--Macaulay $\lL$
for $k \in \{0, 1\}$, since $h_0(\lL) = 1$ and 
$h_1(\lL) = (f_0(\lL) - 2^n)/2^n$, for $k=n$ 
\cite[Corollary~1.2]{Ath12} and for $k = n+1$, since 
$h_{n+1}(\lL) = (-1)^n \widetilde{\chi} (\lL)$,
and for every $k$ in the special case that $\lL$ is the 
cubical barycentric subdivision of a Cohen--Macaulay 
simplicial complex \cite{Het96} (see also 
Remark~\ref{rem:simplicial}).
			
\subsection{Barycentric subdivision and shellability}
\label{sec:shell}

The \emph{barycentric subdivision} of a regular cell
complex $\lL$ is denoted by $\sd(\lL)$ 
and defined as the abstract simplicial complex whose 
faces are the chains $\sigma_0 \subset \sigma_1 
\subset \cdots \subset \sigma_k$ of nonempty faces 
of $\lL$. The natural restriction of $\sd(\lL)$ to a 
nonempty face $\sigma \in \lL$ is exactly the 
barycentric subdivision $\sd(\sigma)$ of (the 
complex of faces of) $\sigma$. 

Similarly, by the barycentric subdivision $\sd(Q)$ 
of a convex polytope $Q$ we mean that of the complex 
of faces 
of $Q$. Since $\sd(Q)$ is a cone over $\sd(\partial 
Q)$, we have $h(\sd(Q), x) = h(\sd(\partial Q), x)$.
For the $n$-dimensional cube $Q$ we have $h(\sd(Q), 
x) = B_n(x)$, where $B_n(x)$ is the Eulerian polynomial
which counts signed permutations of $\{1, 2,\dots,n\}$
by the number of descents of type $B$; see, for 
instance, \cite[Chapter~11]{Pet15}. The following 
well known type $B$ analogue of Worpitzky's identity 
\cite[Equation~(13.3)]{Pet15}
\begin{equation}
\label{eq:Bn-gen}
\frac{B_n(x)}{(1-x)^{n+1}} \ = \ \sum_{m \ge 0} \, 
(2m+1)^n x^m 
\end{equation} 
will make computations in the following section 
easier.

A regular cell complex $\lL$ is called \emph{pure} if 
all its facets (faces which are maximal with respect 
to inclusion) have the same dimension. Such a 
complex $\lL$ is called \emph{shellable} if either 
it is zero-dimensional, or else there exists a 
linear ordering $\tau_1, \tau_2,\dots,\tau_m$ of its 
facets, called a \emph{shelling}, such that (a) 
$\partial \tau_1$ is shellable; and (b) for 
$2 \le j \le m$, the complex of faces of $\partial 
\tau_j$ which are contained in $\tau_1 \cup 
\tau_2 \cup \cdots \cup \tau_{j-1}$ is pure, of the 
same dimension as $\partial \tau_j$, and there exists 
a shelling of $\partial \tau_j$ for which the facets 
of $\partial \tau_j$ contained in $\tau_1 \cup \tau_2 
\cup \cdots \cup \tau_{j-1}$ form an initial segment. 
A fundamental result of Bruggesser and Mani 
\cite{BM71} states that $\partial Q$ is shellable
for every convex polytope $Q$. For the shellability 
of cubical complexes in particular, see 
\cite[Section~3]{EH00} \cite[Section~3]{HS20+}.

\section{The $h$-vector transformation}
\label{sec:h-trans}

This section studies the transformation which maps 
the cubical $h$-vector of a cubical complex $\lL$ to 
the (simplicial) $h$-vector of the barycentric 
subdivision $\sd(\lL)$ and deduces 
Theorem~\ref{thm:main} from its properties. We begin 
with an important definition.
\begin{definition} \label{def:pBnk}
For $n \in \NN$ and $k \in \{0, 1,\dots,n+1\}$ we
denote by $\cC_{n,k}$ the relative simplicial complex 
which is obtained from the barycentric subdivision of 
the $n$-dimensional cube by removing 
\begin{itemize}
\item[$\bullet$] 
no face, if $k=0$, 
\item[$\bullet$] 
all faces which lie in one facet and $k-1$ pairs 
of antipodal facets of the cube (making a total of 
$2k-1$ facets), if $k \in \{1, 2,\dots,n\}$,
\item[$\bullet$] 
all faces on the boundary of the cube, if $k=n+1$.
\end{itemize}
We define $p^B_{n,k}(x) = h(\cC_{n,k}, x)$ for $k 
\in \{0, n+1\}$, and $p^B_{n,k}(x) = 
2 h(\cC_{n,k}, x)$ for $k \in \{1, 2,\dots,n\}$.
\end{definition} 

The polynomials $p^B_{n,k}(x)$ are shown on 
Table~\ref{tab:pBnk} for $n \le 3$. For $n=4$,

\[ p^B_{4,k}(x) \ = \ \begin{cases}
   1 + 76x + 230x^2 + 76x^3 + x^4, 
	 & \text{if $k = 0$,} \\
	 108x + 460x^2 + 196x^3 + 4x^4, 
	 & \text{if $k = 1$,} \\
   36x + 420x^2 + 300x^3 + 12x^4, 
	 & \text{if $k=2$,} \\
   12x + 300x^2 + 420x^3 + 36x^4, 
	 & \text{if $k=3$,} \\
   4x + 196x^2 + 460x^3 + 108x^4, 
	 & \text{if $k=4$,} \\
   x + 76x^2 + 230x^3 + 76x^4 + x^5, 
	 & \text{if $k=5$}. \end{cases}  \]

\medskip
\noindent
Their significance stems from the following theorem.
\begin{theorem} \label{thm:h-trans}
For every $n$-dimensional cubical complex $\lL$, 
\[ h(\sd(\lL), x) \ = \ \sum_{k=0}^{n+1} h_k(\lL) 
   p^B_{n,k}(x) . \]
\end{theorem}

{\scriptsize
\begin{table}[hptb]
\begin{center}
\begin{tabular}{| l || l | l | l | l | l | l ||} \hline
& $k=0$ & $k=1$ & $k=2$    & $k=3$ & $k=4$ \\ \hline \hline
$n=0$   & 1    & $x$    & & & \\ 
         \hline
$n=1$   & $1+x$ & $4x$ & $x+x^2$ & & \\ \hline
 $n=2$  & $1+6x+x^2$ & $12x+4x^2$ & 
          $4x+12x^2$ & $x+6x^2+x^3$ & \\ \hline
 $n=3$  & $1+23x+23x^2+x^3$ & $36x+56x^2+4x^3$ 
        & $12x+72x^2+12x^3$ & $4x+56x^2+36x^3$ & 
				$x+23x^2+23x^3+x^4$ \\ \hline
\end{tabular}
\caption{The polynomials $p^B_{n,k}(x)$ for $n \le 3$.}
\label{tab:pBnk}
\end{center}
\end{table}}

The proof requires a few preliminary results. We
first summarize  some of the main properties of 
$p^B_{n,k}(x)$.
\begin{proposition} \label{prop:pBnk}
For every $n \in \NN$:
\begin{itemize}
\itemsep=0pt
\item[{\rm (a)}]
The polynomial $p^B_{n,k}(x)$ has nonnegative 
coefficients for every $k \in \{0, 1,\dots,n+1\}$;
its degree is equal to $n+1$, if $k = n+1$, and to
$n$ otherwise. 

\item[{\rm (b)}]
$p^B_{n,n+1-k}(x) = x^{n+1}p^B_{n,k}(1/x)$ for
every $k \in \{0, 1,\dots,n+1\}$.

\item[{\rm (c)}]
$p^B_{n,0}(x) = B_n(x)$, $p^B_{n,n+1}(x) = xB_n(x)$ 
and $\sum_{k=0}^{n+1} p^B_{n,k}(x) = B_{n+1}(x)$.

\item[{\rm (d)}]
We have
\[ p^B_{n+1,k+1}(x) \ = \ \begin{cases}
   2 p^B_{n+1,0}(x) + 2 (x-1) p^B_{n,0}(x), 
	 & \text{if $k = 0$,} \\
	 p^B_{n+1,k}(x) + 2(x-1) p^B_{n,k}(x), 
	 & \text{if $1 \le k \le n$,} \\
   (1/2) \cdot p^B_{n+1,n+1}(x) + (x-1) p^B_{n,n+1}(x), 
	 & \text{if $k=n+1$} . \end{cases}  \]

\item[{\rm (e)}]
The recurrence 
\[ p^B_{n+1,k}(x) \ = \ \begin{cases} 
   {\displaystyle \sum_{i=0}^{n+1} p^B_{n,i}(x)}, 
	 & \text{if $k = 0$}, \\
	 {\displaystyle 2x \sum_{i=0}^{k-1} p^B_{n,i}(x) 
	 \, + \, 2 \sum_{i=k}^{n+1} p^B_{n,i}(x)}, 
	 & \text{if $1 \le k \le n+1$}, \\
   {\displaystyle x \sum_{i=0}^{n+1} p^B_{n,i}(x)}, 
	 & \text{if $k=n+2$} \end{cases} \]
holds for $k \in \{0, 1,\dots,n+1\}$. 

\item[{\rm (f)}]
We have 
\[ \frac{p^B_{n,k}(x)}{(1-x)^{n+1}} \ = \ \begin{cases} 
   {\displaystyle \sum_{m \ge 0} \, (2m+1)^n x^m}, 
	 & \text{if $k = 0$,} \\
	 {\displaystyle \sum_{m \ge 0} \, (4m) (2m-1)^{k-1} 
	                (2m+1)^{n-k} x^m}, 
	 & \text{if $1 \le k \le n$,} \\
   {\displaystyle \sum_{m \ge 1} \, (2m-1)^n x^m},
	 & \text{if $k=n+1$}. \end{cases} \]
\end{itemize}
\end{proposition}

\begin{proof} 
We first note that, as discussed in 
Section~\ref{sec:shell}, $p^B_{n,0}(x) = h(\cC_{n,0},x) = 
B_n(x)$. Part (d) follows from Definition~\ref{def:pBnk} 
and the definition of the $h$-polynomial of a relative 
simplicial complex. Indeed, for $1 \le k \le n$, we have
$f(\cC_{n+1,k+1}, x) = f(\cC_{n+1,k}, x) - 2 f(\cC_{n,k}, 
x)$. Hence, 

\begin{eqnarray*}
h(\cC_{n+1,k+1}, x) & = & (1-x)^{n+2} f(\cC_{n+1,k+1}, 
\frac{x}{1-x}) \\ & & \\ & = &
(1-x)^{n+2} f(\cC_{n+1,k}, \frac{x}{1-x}) \, - \, 
2 (1-x) \cdot (1-x)^{n+1} f(\cC_{n,k}, \frac{x}{1-x}) 
\\ & & \\ & = &
h(\cC_{n+1,k}, x) \, + \, 2(x-1) h(\cC_{n,k}, x)
\end{eqnarray*}

\medskip
\noindent
and

\begin{eqnarray*}
p^B_{n+1,k+1}(x) & = & 2 h(\cC_{n+1,k+1}, x) 
\ = \ 2 h(\cC_{n+1,k}, x) \, + \, 4 (x-1) 
h(\cC_{n,k}, x) \\ & = & p^B_{n+1,k}(x) \, + 
\, 2 (x-1) p^B_{n,k}(x).
\end{eqnarray*}

\medskip
\noindent
The same argument, similar to that in the proof of 
\cite[Corollary~5.6]{Ath20+}, works for $k \in \{0, 
n+1\}$. Part (f) follows from part (d) by 
straightforward induction on $k$ (for fixed $n$), 
where the base $k=0$ of the induction holds because 
of Equation~(\ref{eq:Bn-gen}).

For part (c) we first note that $p^B_{n,n+1}(x) = 
h(\cC_{n,n+1},x) = h^\circ(\cC_{n,0}, x) = x^{n+1} 
h(\cC_{n,0}, 1/x) = x^{n+1} B_n(1/x) = xB_n(x)$. The 
identity for the sum of the $p^B_{n,k}(x)$ can be 
verified directly by summing that of part (f). For 
a more conceptual proof, one can use an obvious 
shelling of the boundary complex of 
the $(n+1)$-dimensional cube to write, as explained 
in \cite[Section~3]{HS20+}, the $h$-polynomial 
$B_{n+1}(x)$ of its barycentric subdivision as a 
sum of $h$-polynomials of relative simplicial 
complexes, each one combinatorially isomorphic to 
one of the $\cC_{n,k}$. The details are left to 
the interested reader.

Given (c), the recursion of part (e) follows easily
by induction on $k$ from part (d) (this parallels
the proof of \cite[Lemma~6.3]{Ath20+}).

Part (b) is a consequence of Definition~\ref{def:pBnk}
and Proposition~\ref{prop:hsymmetry}. Alternatively, 
it follows from part (f) and standard facts 
about rational generating functions; see 
\cite[Proposition~4.2.3]{StaEC1}. The nonnegativity
of the coefficients of $p^B_{n,k}(x)$, claimed in part
(a), follows from the recursion of part (e), as well 
as from general results \cite[Corollary~III.7.3]{StaCCA} 
on the nonnegativity of $h$-vectors of Cohen--Macaulay
relative simplicial complexes. The statement about the 
degree of $p^B_{n,k}(x)$, claimed there, follows from 
either of parts (d), (e) or (f).
\end{proof}

We leave the problem to find a combinatorial 
interpretation of $p^B_{n,k}(x)$ open. Given part
(c) of the proposition, one naturally expects that
there is such an interpretation which refines one 
of the known combinatorial interpretations of 
$B_{n+1}(x)$.

The following statement is a consequence of a more
general result \cite[Proposition~7.6]{Sta92} of 
Stanley on subdivisions of CW-posets. To keep this 
paper self-contained, we include a proof.
\begin{proposition} \label{prop:sdc-h}
For every $n$-dimensional cubical complex $\lL$, 
\[ h(\sd(\lL), x) \ = \ (1-x)^{n+1} \, + \, x \, 
   \sum_{k=0}^n f_k(\lL) (1-x)^{n-k} B_k(x) . \]
\end{proposition}
\begin{proof} 
Since every face of $\sd(\lL)$ is an interior face 
of the restriction $\sd(\sigma)$ of $\sd(\lL)$ to 
a unique face $\sigma \in \lL$, we have
\[ f(\sd(\lL), x) \ = \ \sum_{\sigma \in \lL} f^\circ
   (\sd(\sigma), x) \ = \ 1 \, + \sum_{\sigma \in \lL 
	 \sm \{\varnothing\}} f^\circ(\sd(\sigma), x) . \]
Transforming $f$-polynomials to $h$-polynomials in
this equation and recalling from Section~\ref{sec:enu} 
that $h^\circ(\sd(\sigma), x) = x^{k+1} h(\sd(\sigma), 
1/x) = x^{k+1} B_k(1/x) = x B_k(x)$ for every nonempty 
$k$-dimensional face $\sigma \in \lL$, we get

\begin{eqnarray*}
h(\sd(\lL), x) & = & (1-x)^{n+1} f(\sd(\lL), 
\frac{x}{1-x}) \\ & = & (1-x)^{n+1} \ + 
\sum_{\sigma \in \lL \sm \{\varnothing\}} (1-x)^{n+1} 
f^\circ(\sd(\sigma), \frac{x}{1-x}) \\ & = & 
(1-x)^{n+1} \ + \sum_{\sigma \in \lL \sm \{\varnothing\}} 
(1-x)^{n-\dim(\sigma)} \, h^\circ(\sd(\sigma), x) \\ & = & 
(1-x)^{n+1} \, + \ \sum_{k=0}^n f_k(\lL) (1-x)^{n-k} x 
B_k(x) 
\end{eqnarray*}
and the proof follows.
\end{proof}

\medskip
\noindent
\emph{Proof of Theorem~\ref{thm:h-trans}}. Let us 
denote by $p(\lL, x)$ the right-hand side of the desired
equality. Clearly, it suffices to show that $h(\sd(\lL), 
x)/(1-x)^{n+1} = p(\lL, x)/(1-x)^{n+1}$. From 
Proposition~\ref{prop:sdc-h} and 
Equation~(\ref{eq:Bn-gen}) we deduce that 
\begin{eqnarray*} 
\frac{h(\sd(\lL), x)}{(1-x)^{n+1}} & = & 1 \, + \, x \, 
\sum_{k=0}^n f_k(\lL) \, \frac{B_k(x)}{(1-x)^{k+1}} \ = 
\ 1 \, + \, \sum_{m \ge 0} \left( \, \sum_{k=0}^n 
f_k(\lL) (2m+1)^k \right) x^{m+1} \nonumber \\ & = & 
1 \, + \, \sum_{m \ge 1} \left( \, \sum_{k=0}^n 
f_k(\lL) (2m-1)^k \right) x^m .
\end{eqnarray*} 

\medskip
Similarly, from part (f) of Proposition~\ref{prop:pBnk}
we get
\[ \frac{p(\lL, x)}{(1-x)^{n+1}} \ = \ \sum_{k=0}^{n+1} 
h_k(\lL) \, \frac{p^B_{n,k}(x)}{(1-x)^{n+1}} \ = \ 1 \, + 
\, \sum_{m \ge 1} a_\lL(m) x^m, \]
where 
\[ a_\lL(y) \ := \ h_0(\lL) (2y+1)^n \, + \, \sum_{k=1}^n 
h_k(\lL) (4y) (2y-1)^{k-1} (2y+1)^{n-k} \, + \, h_{n+1}(\lL) 
(2y-1)^n. \]

Thus, it remains to show that 
\begin{equation} \label{eq:final}
\sum_{k=0}^n f_k(\lL) (2y-1)^k \ = \ a_\lL(y) . 
\end{equation}
We claim that this is, essentially, the defining 
Equation~(\ref{def:cub-h}) of the cubical $h$-polynomial 
of $\lL$ in disguised form. Indeed, cancelling first 
the summand $1 + h_{n+1}(\lL) x^{n+2} = 1 + (-1)^n
\widetilde{\chi} (\lL) x^{n+2}$, and then a factor of 
$x$, from both sides of (\ref{def:cub-h}) gives that
\[ \sum_{k=0}^n (h_k(\lL) + h_{k+1}(\lL)) x^k \ = \ 
   \left( \frac{1-x}{2} \right)^n \, \sum_{k=0}^n 
   f_k(\lL) \left( \frac{2x}{1-x} \right)^ k . \]
Setting $2x/(1-x) = 2y-1$, so that $x = (2y-1)/(2y+1)$ 
and $(1-x)/2 = 1/(2y+1)$, the previous identity can 
be rewritten as 
\begin{equation} \label{def:cub-h2}
\sum_{k=0}^n f_k(\lL) (2y-1)^k \ = \ \sum_{k=0}^n 
\, (h_k(\lL) + h_{k+1}(\lL)) (2y-1)^k (2y+1)^{n-k} .
\end{equation}
Since the right-hand side is readily equal to 
$a_\lL(y)$, this proves Equation~(\ref{eq:final}) 
and the theorem as well.
\qed

\medskip
To deduce Theorem~\ref{thm:main} from 
Theorem~\ref{thm:h-trans} and 
Proposition~\ref{prop:pBnk}, we need to recall a few 
definitions and facts from the theory of 
interlacing polynomials; for more information, see 
\cite[Section~8]{Bra15} and references therein. A 
polynomial $p(x) \in \RR[x]$ is called 
\emph{real-rooted} if either it is the zero 
polynomial, or every complex 
root of $p(x)$ is real. Given two real-rooted 
polynomials $p(x), q(x) \in \RR[x]$, we say that 
$p(x)$ \emph{interlaces} $q(x)$ if the roots 
$\{\alpha_i\}$ of $p(x)$ interlace (or alternate 
to the left of) the roots $\{\beta_j\}$ of $q(x)$, 
in the sense that they can be listed as
\[ \cdots \le \beta_3 \le \alpha_2 \le \beta_2 \le 
       \alpha_1 \le \beta_1 \le 0. \]

A sequence $(p_0(x), p_1(x),\dots,p_n(x))$ of 
real-rooted polynomials is called \emph{interlacing} 
if $p_i(x)$ interlaces $p_j(x)$ for all $0 \le i < j 
\le n$. Assuming also that these polynomials have 
positive leading coefficients, every nonnegative 
linear combination of $p_0(x), p_1(x),\dots,p_n(x)$ 
is real-rooted and interlaced by $p_0(x)$. A standard 
way to produce interlacing sequences in combinatorics 
is the following. Suppose that $p_0(x), 
p_1(x),\dots,p_n(x)$ are real-rooted polynomials 
with nonnegative coefficients and set 
\[ q_k(x) \ = \ x \sum_{i=0}^{k-1} p_i(x) \, + \, 
   \sum_{i=k}^n p_i(x) \]
for $k \in \{0, 1,\dots,n+1\}$. Then, if the sequence 
$(p_0(x), p_1(x),\dots,p_n(x))$ is interlacing, so is 
$(q_0(x), q_1(x),\dots,q_{n+1}(x))$; see 
\cite[Corollary~8.7]{Bra15} for a more general 
statement.

\smallskip
The following result is a stronger version of 
Theorem~\ref{thm:main}.
\begin{corollary} \label{cor:main}
The polynomial $h(\sd(\lL), x)$ is real-rooted and 
interlaced by the Eulerian polynomial $B_n(x)$ for 
every $n$-dimensional cubical complex $\lL$ with 
nonnegative cubical $h$-vector. 

In particular, $h(\sd(\lL), x)$ and $h(\sd(Q), x)$
are real-rooted and interlaced by $B_n(x)$ for every 
shellable, $n$-dimensional cubical complex $\lL$ and  
every cubical polytope $Q$ of dimension $n+1$, 
respectively.
\end{corollary}

\begin{proof}
By an application of the lemma on interlacing 
sequences just discussed, the recurrence of part 
(e) of Proposition~\ref{prop:pBnk} implies that 
$(p^B_{n,0}(x), p^B_{n,1}(x),\dots,p^B_{n,n+1}(x))$ 
is interlacing for every $n \in \NN$ by induction 
on $n$. Therefore, being a nonnegative linear 
combination of the elements of the sequence by 
Theorem~\ref{thm:h-trans}, $h(\sd(\lL), x)$ is 
real-rooted and interlaced by $p^B_{n,0}(x) = 
B_n(x)$ for every $n$-dimensional cubical complex 
$\lL$ with nonnegative cubical $h$-vector. This
proves the first statement.

The second statement follows from the first since 
shellable cubical complexes are known to have 
nonnegative cubical $h$-vector 
\cite[Theorem~5~(iii)]{Ad96}, $h(\sd(Q), x) = 
h(\sd(\partial Q), x)$ for every convex polytope 
$Q$ and because boundary complexes of polytopes
are shellable.
\end{proof}

\begin{remark} \label{rem:simplicial} \rm
Let $\Delta$ be a simplicial complex with nonnegative 
$h$-vector and $\lL$ be a cubical complex which is 
obtained from $\Delta$ by any operation which 
preserves nonnegativity of $h$-vectors.
Corollary~\ref{cor:main} implies that $h(\sd(\lL), 
x)$ is real-rooted. 

By a result of 
Hetyei~\cite{Het96}, such an operation is the 
cubical barycentric subdivision $\lL = \sd_c(\Delta)$
(see \cite[p.~44]{Ath18}), also known as barycentric 
cover \cite[Section~2.3]{BBC97}, of $\Delta$. Then,
$\sd(\lL)$ becomes the interval triangulation 
of $\Delta$ \cite[Section~3.3]{MW17}. This argument 
shows that the interval triangulation of $\Delta$ 
has a real-rooted $h$-polynomial for every simplicial 
complex $\Delta$ with nonnegative $h$-vector and
answers in the affirmative the question 
of \cite[Problem~33]{MW17}. Although there are 
other proofs of this fact in the literature (see
\cite{Ath20+} and references therein), the approach 
via Corollary~\ref{cor:main} allows for 
more general results, e.g., by applying further 
cubical subdivisions of $\sd_c(\Delta)$ which 
preserve the nonnegativity of the cubical 
$h$-vector. 
\end{remark}

\begin{remark} \label{rem:BW-formula} \rm
Applying the reasoning of the proof of 
Proposition~\ref{prop:sdc-h} and of the first few lines
of the proof of Theorem~\ref{thm:h-trans} to an 
$n$-dimensional simplicial complex $\Delta$ gives 
that
\[ h(\sd(\Delta), x) \ = \ (1-x)^{n+1} \, + \, x \, 
   \sum_{k=0}^n f_k(\Delta) (1-x)^{n-k} A_{k+1}(x) \]
and 
\[
\frac{h(\sd(\Delta), x)}{(1-x)^{n+2}} \ = \
\sum_{m \ge 0} \left( \, \sum_{k=0}^{n+1} f_{k-1}
(\Delta) m^k \right) x^m \ = \ \sum_{m \ge 0} 
\left( \, \sum_{k=0}^{n+1} h_k(\Delta) m^k 
(m+1)^{n+1-k} \right) x^m . \]
This is the expression at which Brenti and Welker 
arrived \cite[Equation~(5)]{BW08} via a different 
route and which they used to show that $h(\sd(\Delta), 
x)$ has only real roots, provided that $h_k(\Delta) 
\ge 0$ for all $k$.
\end{remark}

\begin{remark} \rm
Replacing $2y-1$ by $x$ in (\ref{def:cub-h2})
shows that the equation 
\[ \sum_{k=0}^n f_k(\lL) x^k \ = \ \sum_{k=0}^n 
\, (h_k(\lL) + h_{k+1}(\lL)) \, x^k (x+2)^{n-k} , \]
together with the condition $h_0(\lL) = 1$, gives
an equivalent way to define the normalized cubical 
$h$-vector of an $n$-dimensional cubical complex $\lL$.
\end{remark}

\section{Closing remarks}
\label{sec:rem}

There is a large literature on the barycentric 
subdivision of simplicial complexes which relates 
to the work \cite{BW08}. Many of the questions 
addressed there make sense for cubical complexes. 
We only consider a couple of them here.

\medskip
\textbf{a}. Being real-rooted, $h(\sd(\Delta), x)$ 
is unimodal for every $n$-dimensional simplicial 
complex $\Delta$ with nonnegative $h$-vector. 
Kubitzke and Nevo showed~\cite[Corolalry~4.7]{KN09} 
that the corresponding $h$-vector 
$(h_i(\sd(\Delta))_{0 \le i \le n+1}$ has a peak 
at $i = (n+1)/2$, if $n$ is odd, and at $i = n/2$ 
or $i = n/2 + 1$, if $n$ is even. The analogous 
statement for cubical complexes follows from 
Theorem~\ref{thm:h-trans} since, as in the 
simplicial setting, the 
unimodal polynomial $p^B_{n,k}(x)$ has a peak at 
$i = (n+1)/2$, if $n$ is odd, at $i = n/2$ if $n$ 
is even and $k \le n/2$, and at $i = n/2 + 1$, if 
$n$ is even and $k \ge n/2 + 1$. The latter claim 
can be deduced from the recursion 
of part (e) of Proposition~\ref{prop:pBnk} by 
mimicking the argument given in the simplicial 
setting in~\cite[Section~2]{Mur10}. For general 
results on the unimodality of $h$-vectors of 
barycentric subdivisions of Cohen--Macaulay 
regular cell complexes, proven by algebraic methods,
see Corollaries~1.2 and~5.12 in \cite{MY14}.

\textbf{b}. The main result of~\cite{BS20} implies 
(see~\cite[Section~8]{Ath20+}) that 
$h(\sd(\Delta), x)$ has a nonnegative real-rooted 
symmetric decomposition with respect to $n$ for 
every triangulation $\Delta$ of an $n$-dimensional 
ball. Does this hold if $\Delta$ is replaced by 
any cubical subdivision of the $n$-dimensional ball? 
Are these symmetric decompositions interlacing? Do
the polynomials $p^B_{n,k}(x)$ have such 
properties?

\textbf{c}. The subdivision operator (see 
\cite[Section~3.3]{Bra15}) has a natural generalization 
in the context of uniform triangulations of simplicial
complexes \cite[Section~5]{Ath20+} which plays a role
in that theory. It may be worth studying the cubical 
analogue of this operator further.


\end{document}